\documentclass{article}

\usepackage{amsmath,amssymb,amsthm}
 
 \theoremstyle{plain}
 \newtheorem{thm}{Theorem}
 \newtheorem{prop}[thm]{Proposition}
 \newtheorem{lem}[thm]{Lemma}

\newcommand{\bbQ}{\ensuremath{\mathbb{Q}}}
\newcommand{\bbC}{\ensuremath{\mathbb{C}}}
\newcommand{\bbP}{\ensuremath{\mathbb{P}}}
\newcommand{\PGL}{\ensuremath{\operatorname{PGL}}}
\newcommand{\GL}{\ensuremath{\operatorname{GL}}}
\newcommand{\ed}{\ensuremath{\operatorname{ed}}}
\newcommand{\rank}{\ensuremath{\operatorname{rank}}}

\begin{document}
 
\title{Essential Dimensions of $A_7$ and $S_7$}
\author{Alexander Duncan \thanks{The author is partially supported by an NSERC Canada Graduate Scholarship.}}

\maketitle

\abstract{
We show that Y. Prokhorov's ``Simple Finite Subgroups of the {C}remona Group of Rank 3'' implies that, over any field of characteristic $0$, the essential dimensions of the alternating group, $A_7$, and the symmetric group, $S_7$, are $4$.
}

\subsection*{Introduction}

Let $k$ be a field of characteristic $0$.  Throughout this note we assume that all varieties, actions and maps are defined over $k$.

Let $G$ be a finite group.  A \emph{compression} is a dominant rational $G$-equivariant map of faithful $G$-varieties.   Let $V$ be a faithful linear representation of $G$ viewed as a $G$-variety.  We define the \emph{essential dimension of $G$}, denoted $\ed_k(G)$, to be the minimal value of $\dim(X)$, where $X$ is taken from the set of all faithful $G$-varieties sitting under a compression $V \dasharrow X$.  From \cite[Theorem 3.1]{BuhRei1997EDFG}, we see that the essential dimension depends only on $k$ and $G$ --- the choice of linear representation $V$ does not matter.

The purpose of this note is to show that the essential dimension of the alternating group $A_7$ and the symmetric group $S_7$ can be computed using the recent work of Prokhorov \cite{Pro2009SFSCGR} on the classification of rationally connected threefolds with faithful actions of non-abelian simple groups.  Our main result is the following:

\begin{thm}\label{thm:main}
$\ed_k(A_7)=\ed_k(S_7)=4$.
\end{thm}

The essential dimension of a finite group was introduced by Buhler and Reichstein in \cite{BuhRei1997EDFG}.  The concept has since been extended to much broader contexts (see \cite{Rei2000NEDAG} and \cite{BerFav2003EDFPVAAM}).

The following results hold when $k$ has ``sufficiently many roots of unity;'' for example, when $k$ is algebraically closed.  If $G$ is an abelian group then $\ed_k(G)=\rank(G)$ \cite[Theorem 6.1]{BuhRei1997EDFG}.  We have $\ed_k(G)=1$ if and only if $G$ is cyclic or odd dihedral \cite[Theorem 6.2]{BuhRei1997EDFG}; see also \cite{Led2007FGED} and \cite{ChuHuKanZha2008GWED}.  If $G$ is a $p$-group then $\ed_k(G)$ is equal to the minimal dimension of a faithful linear representation of $G$; this is a deep result of Karpenko and Merkurjev \cite{KarMer2008EDFP}.

The values of $\ed_k(S_n)$ are of special interest because they relate to classical questions of simplifying degree $n$ polynomials via Tschirnhaus transformations.  In particular, the degree $7$ case features prominently in algebraic variants of Hilbert's 13th problem.  In this language, several results for small $n$ were established by Hermite, Joubert and Klein in the 1800s. For more information, see the discussion in \cite{BuhRei1997EDFG} or \cite{BuhRei1999TT}.

The values of $\ed_k(S_n)$ and $\ed_k(A_n)$ are known for all $n \le 6$.  Buhler and Reichstein \cite{BuhRei1997EDFG} establish bounds for symmetric groups when $n \ge 5$:
\begin{equation}\label{eqn:origSnBounds}
n-3 \ge \ed_k(S_n) \ge \lfloor n/2 \rfloor \,.
\end{equation}
We note that these bounds tell us that $\ed_k(S_7)$ is either $3$ or $4$.  For the alternating groups $A_n$, they found the following bounds when $n \ge 5$:
\begin{equation}\label{eqn:origAnBounds}
n-3 \ge \ed_k(A_n) \ge 2\lfloor n/4 \rfloor \,.
\end{equation}
From this $\ed_k(A_6)$ is either $2$ or $3$.  Recently, Serre found the exact value:

\newcommand{\workaroundA}{Serre \cite[Proposition 3.6]{Ser2008GCSF}}
\begin{thm}[\workaroundA]\label{thm:Serre}
$\ed_k(A_6)=3$.
\end{thm}
  
Taking Theorems \ref{thm:main} and \ref{thm:Serre} into account we can improve some of the known bounds in higher dimensions.  From \cite[Theorem 6.5]{BuhRei1997EDFG}, we have that $\ed_k(S_{n+2}) \ge \ed_k(S_{n})+1$ for any $n \ge 1$.  Similarly, from \cite[Theorem 6.7]{BuhRei1997EDFG}, we have $\ed_k(A_{n+4}) \ge \ed_k(A_{n})+2$ for any $n \ge 4$.  We have the following for $n \ge 6$:
\begin{align}
n-3 \ge \ed_k(S_n) &\ge \left\lfloor \frac{n+1}{2} \right\rfloor \,, \\
n-3 \ge \ed_k(A_n) &\ge
\begin{cases}
\frac{n}{2} & \textrm{ for } n \textrm{ even}\\
\frac{n-1}{2} & \textrm{ for } n \equiv 1\bmod{4}\\
\frac{n+1}{2} & \textrm{ for } n \equiv 3\bmod{4}
\end{cases}\,.
\end{align}

\subsection*{Proof of the main theorem}

Our proof of Theorem \ref{thm:main} is in the same spirit as Serre's proof of Theorem \ref{thm:Serre}.  For Serre's argument, it suffices to show $\ed_k(A_6) \ne 2$ by the bounds in (\ref{eqn:origAnBounds}).  One must show no $A_6$-surface sits under a compression from a linear $A_6$-variety.  Serre uses the Enriques-Manin-Iskovskikh classification of minimal rational $G$-surfaces (see \cite{Man1967RSOPF} and \cite{Isk1979MMRSOAF}) to reduce the problem to one surface with an $A_6$-action ($\bbP^2$ with the linear action).  It is then shown that the group acting on this remaining surface has an abelian subgroup without fixed points.  This eliminates this last surface in view of the following:

\newcommand{\workaroundB}{\cite[Proposition A.2]{ReiYou2000EDAGRTG}}
\begin{prop}[\workaroundB]\label{prop:goingdown}
Let $A$ be a finite abelian group and $\psi : V \dasharrow X$ be an $A$-equivariant rational map of $A$-varieties over $\bbC$.  If $V$ has a smooth $A$-fixed point and $X$ is proper then $X$ has an $A$-fixed point.
\end{prop}

For our proof, will need to show that $\ed_\bbC(A_7) \ne 3$.  Serre looked at rational surfaces; we consider unirational threefolds.  Our analog of Serre's reduction to $\bbP^2$ is Prokhorov's classification for the group $A_7$:

\newcommand{\workaroundC}{Prokhorov \cite[Theorem 1.5]{Pro2009SFSCGR}}
\begin{thm}[\workaroundC]\label{thm:Prokhorov}
Let $X$ be a rationally connected threefold over $\bbC$ with a faithful action of $A_7$.  Then $X$ is equivariantly birationally equivalent to one of the following:
\begin{enumerate}
\item[(i)] A subvariety of $\bbP^6$, with the standard permutation $A_7$ action, cut out by symmetric polynomials of degrees $1$, $2$ and $3$.
\item[(ii)] $\bbP^3$ with a linear action of $A_7$.
\end{enumerate}
\end{thm}

We will need the following lemma to reduce to the case where $k=\bbC$:

\begin{lem}\label{lem:Csuffices}
Suppose $G$ is a finite group and $k$ is a field of characteristic $0$.  Then $\ed_k(G) \ge \ed_\bbC(G)$.
\end{lem}

\begin{proof}
First, note that $\ed_k(G) \ge \ed_K(G)$ for $K$ an algebraic closure of $k$ (see \cite[Proposition 1.5]{BerFav2003EDFPVAAM}).  Next, we have $\ed_K(G) = \ed_\bbC(G)$ since $K$ and $\bbC$ both contain an algebraic closure of $\bbQ$ (see \cite[Proposition 2.14(1)]{BroReiVis2007EDAS}).
\end{proof}

\begin{proof}[Proof of Theorem \ref{thm:main}]
We have the following string of inequalities:
\[4 \ge \ed_k(S_7) \ge \ed_k(A_7) \ge \ed_\bbC(A_7) \ge \ed_\bbC(A_6) = 3\,.\]
Indeed, the first inequality follows from the bound \ref{eqn:origSnBounds}.  The second and fourth inequalities follow from the standard fact that $\ed_k(G) \ge \ed_k(H)$ for any subgroup $H$ of a finite group $G$.  The third inequality follows from Lemma \ref{lem:Csuffices}.  To prove the theorem it suffices to prove that $\ed_\bbC(A_7) \ne 3$.

Suppose $\ed_\bbC(A_7) = 3$.  Then there exists a dominant rational $A_7$-equivariant map $\psi : V \dasharrow X$ from a linear $A_7$-variety $V$ to a $3$-dimensional $A_7$-variety $X$.  From this, $X$ is unirational and, thus, rationally connected.  We may assume that $X$ is one of the threefolds from Prokhorov's Theorem.

Note that $V$ has an $A_7$-fixed point (the origin) and $X$ is proper.  Thus all abelian subgroups of $A_7$ have fixed points by Proposition \ref{prop:goingdown}.  For each threefold, we will exhibit an abelian subgroup of $A_7$ without fixed points on $X$.   This leads to a contradiction and, so, $\ed_\bbC(A_7) \ne 3$ as desired.

\smallskip\noindent
{\bf Case (i):}  Consider $A = \langle (1\ 2\ 3), (4\ 5\ 6) \rangle$, an abelian subgroup of $A_7$.  Let $\zeta$ be a third root of unity.  Consider the following points in $\bbP^6$:

\begin{center}
$(\lambda_1 : \lambda_1 : \lambda_1 : \lambda_2 : \lambda_2 : \lambda_2 : \lambda_3)$ \\
$(1 : \zeta : \zeta^2 : 0 : 0 : 0 : 0)$ \\
$(1 : \zeta^2 : \zeta : 0 : 0 : 0 : 0)$ \\
$(0 : 0 : 0 : 1 : \zeta : \zeta^2 : 0)$ \\
$(0 : 0 : 0 : 1 : \zeta^2 : \zeta : 0)$
\end{center}

where $\lambda_1,\lambda_2,\lambda_3 \in \bbC$ are not all $0$.  These correspond to the eigenspaces of a lift of $A$ acting on $\bbC^7$.  Thus these are all the $A$-fixed points on $\bbP^6$.

We claim that none of these points lie on $X$.  For points of the first form, there are only two solutions of $x_1 + \ldots + x_7 = 0$ and $x_1^2 + \ldots + x_7^2 = 0$:
\[ \lambda_1 = -1 \pm \sqrt{-7}, \quad \lambda_2 = -1 \mp \sqrt{-7}, \quad \lambda_3 = 6 \]
One then checks that $x_1^3 + \ldots + x_7^3 \ne 0$ for these two points and for the remaining points.  We have an abelian subgroup without fixed points --- a contradiction.

\smallskip\noindent
{\bf Case (ii):}  In this case $A_7$ acts linearly on $\bbP^3$ and can be viewed as a subgroup of $\PGL_4(\bbC)$.  Let
\[ A = \langle (1\ 2)(3\ 4), (1\ 2)(5\ 6) \rangle \]
be an abelian subgroup of $A_7$.  Let $B$ be the inverse image of $A$ in $\GL_4(\bbC)$.  We have the following exact sequence of groups:
\[ 1 \to \bbC^\times \to B \to A \to 1 \]
where $\bbC^\times$ is the set of scalar matrices in $\GL_4(\bbC)$.  Recall that $A$ has a fixed point on $\bbP^3$.  This is equivalent to saying that the action of $B$ (viewed as a $4$-dimensional linear representation) has a $1$-dimensional subrepresentation $\chi : B \to \bbC^\times$.  This gives us a splitting $B \simeq A \times \bbC^\times$.  In particular, $B$ is abelian.

There are two distinct projective representations of $A_7$ in $PGL_4(\bbC)$ \cite{ConCurNorPar1985AFG}, but they are complex conjugates so it suffices to look at one.  An explicit description of a preimage of $A_7$ in $\GL_4(\bbC)$ can be found in Blichfeldt \cite[pg. 142]{Bli1917FCG}.  Using a computer algebra package, one checks that elements in the preimages of $(1\ 2)(3\ 4)$ and $(1\ 2)(5\ 6)$ do not commute.  Thus $B$ is not abelian --- a contradiction.
\end{proof}

\subsection*{Acknowledgements}
The author would like to thank Y. Prokhorov for helpful correspondence and Z. Reichstein for extensive comments on earlier versions of this note.

\bibliographystyle{plain}
\bibliography{bibliography}

\begin{thebibliography}{10}

\bibitem{BerFav2003EDFPVAAM}
Gr{\'e}gory Berhuy and Giordano Favi.
\newblock Essential dimension: a functorial point of view (after {A}.
  {M}erkurjev).
\newblock {\em Doc. Math.}, 8:279--330 (electronic), 2003.

\bibitem{Bli1917FCG}
H.~F. Blichfeldt.
\newblock {\em Finite Collineation Groups}.
\newblock The University of Chicago Press, 1917.

\bibitem{BroReiVis2007EDAS}
Patrick Brosnan, Zinovy Reichstein, and Angelo Vistoli.
\newblock Essential dimension and algebraic stacks.
\newblock arXiv:math/0701903v1 [math.AG], 2007.

\bibitem{BuhRei1997EDFG}
J.~Buhler and Z.~Reichstein.
\newblock On the essential dimension of a finite group.
\newblock {\em Compositio Math.}, 106(2):159--179, 1997.

\bibitem{BuhRei1999TT}
Joe Buhler and Zinovy Reichstein.
\newblock On {T}schirnhaus transformations.
\newblock In {\em Topics in number theory (University Park, PA, 1997)}, volume
  467 of {\em Math. Appl.}, pages 127--142. Kluwer Acad. Publ., Dordrecht,
  1999.

\bibitem{ChuHuKanZha2008GWED}
Huah Chu, Shou-Jen Hu, Ming-Chang Kang, and Jiping Zhang.
\newblock Groups with essential dimension one.
\newblock {\em Asian J. Math.}, 12(2):177--191, 2008.

\bibitem{ConCurNorPar1985AFG}
J.~H. Conway, R.~T. Curtis, S.~P. Norton, R.~A. Parker, and R.~A. Wilson.
\newblock {\em Atlas of finite groups}.
\newblock Oxford University Press, Eynsham, 1985.
\newblock Maximal subgroups and ordinary characters for simple groups, With
  computational assistance from J. G. Thackray.

\bibitem{Isk1979MMRSOAF}
V.~A. Iskovskih.
\newblock Minimal models of rational surfaces over arbitrary fields.
\newblock {\em Izv. Akad. Nauk SSSR Ser. Mat.}, 43(1):19--43, 237, 1979.
\newblock English translation: Math. USSR-Izv. 14 (1980), no. 1, 17--39.

\bibitem{KarMer2008EDFP}
Nikita~A. Karpenko and Alexander~S. Merkurjev.
\newblock Essential dimension of finite {$p$}-groups.
\newblock {\em Invent. Math.}, 172(3):491--508, 2008.

\bibitem{Led2007FGED}
Arne Ledet.
\newblock Finite groups of essential dimension one.
\newblock {\em J. Algebra}, 311(1):31--37, 2007.

\bibitem{Man1967RSOPF}
Ju.~I. Manin.
\newblock Rational surfaces over perfect fields. {II}.
\newblock {\em Mat. Sb. (N.S.)}, 72 (114):161--192, 1967.

\bibitem{Pro2009SFSCGR}
Yuri Prokhorov.
\newblock Simple finite subgroups of the {C}remona group of rank 3.
\newblock arXiv:0908.0678v1 [math.AG], 2009.

\bibitem{Rei2000NEDAG}
Z.~Reichstein.
\newblock On the notion of essential dimension for algebraic groups.
\newblock {\em Transform. Groups}, 5(3):265--304, 2000.

\bibitem{ReiYou2000EDAGRTG}
Zinovy Reichstein and Boris Youssin.
\newblock Essential dimensions of algebraic groups and a resolution theorem for
  {$G$}-varieties.
\newblock {\em Canad. J. Math.}, 52(5):1018--1056, 2000.
\newblock With an appendix by J\'anos Koll\'ar and Endre Szab\'o.

\bibitem{Ser2008GCSF}
Jean-Pierre Serre.
\newblock Le group de cremona et ses sous-groupes finis.
\newblock {\em S\'eminaire Bourbaki}, 2008.

\end{thebibliography}

\end{document}